\documentclass{article}
\usepackage[T2A]{fontenc}
\usepackage[english]{babel}
\usepackage{indentfirst}
\usepackage{amsfonts,amsmath,amssymb}
\usepackage{amsthm}
\usepackage[dvips]{graphicx}
\usepackage{array}

\usepackage{esint}

\newtheorem{Thm}{Theorem}
\newtheorem{Lem}{Lemma}
\newtheorem{prp}{Property}
\newtheorem{Rem}{Remark}
\newtheorem{dfn}{Definition}

\begin{document}

MSC 37H05, 60K37

\begin{center}
{Valery Doobko\footnote{Research Center of the Institute of Cybernetics, Kiev, Ukraine; doobko2017@ukr.net} and Elena Karachanskaya\footnote{Far Eastern State Transport University, Khabarovsk, Russia; elena$_{-}$chal@mail.ru}}
\end{center}

\begin{center}
\textbf {Indicator Random Processes and its Application for Modeling Open Stochastic Systems}
\end{center}

\begin{abstract}
The authors present a method of indicator random processes, applicable to constructing models of jump processes associated with diffusion process. Indicator random processes are processes that take only two values: 1 and 0, in accordance with some probabilistic laws. It is shown that the indicator random process is invariant when reduced to an arbitrary positive degree. Equations with random coefficients used in modeling dynamic systems, when applying the method of indicator random processes, can take into account the possibility of adaptation to external changes, including random ones, in order to preserve indicators important for the existence of the system, which can be continuous or discrete. In the case of indicator random processes, defined as functions of the Poisson process, equations for dynamic processes in a media with abruptly changing properties are constructed and studied. To study the capabilities of the proposed method, dynamic models of the diffusion process in media with delay centers and diffusion processes during transitions by switching from one subspace to another were studied. For these models, equations for characteristic functions are constructed. Using the method of indicator random processes, a characteristic function for the Kac model was constructed. It is shown that in the case of dependence of the indicator random process on the Poisson process, the equation for the characteristic function corresponds to the telegraph equation. This result coincides with the result of Kac.
\end{abstract}

Keywords: Indicator random processes; stochastic dynamical models.

\section{Introduction}
An open system is a system that has external interactions. Therefore, a process modeling in an open system is difficult due to external random disturbances that greatly affect the process.The modeling of diffusion processes with jumps is still a pressing problem, since many applied problems are described by similar models. Sharp fluctuations in option prices have led to the need to study jump diffusion models, for example, the Merton model%  \cite{ref-journal01}, \cite{ref-journal02}
, and Kou's model% \cite{ref-book09}
.

Stochastic processes that can instantly, due to a jump, change their parameters, described by systems of stochastic differential equations with switching, are also used in problems of financial mathematics% \cite{ref-proceeding02}, \cite{ref-journal03}, \cite{ref-journal04}, and others
. Switching jump-diffusion models are also used for problems in control theory, and diffusion in random media \cite{ref-book-hanson} %\cite{ref-journal12}, \cite{ref-journal13}, \cite{ref-journal14}
, \cite{ref-journal18}, \cite{ref-journal17}, and others.

In contrast to the models mentioned above, we will consider the It\^{o} stochastic differential equations with regime switching based on the method, which we propose below. These equations can be used in the simulation of dynamic systems that can adapt to external changes, including random ones. These changes in the structure of the system can be continuous, discrete or mixed. If changes in the coefficients of the equations is jump-like, then one of the ways to model transitions from one state of the system to another is to use methods from the theory of queuing systems. Nevertheless, to describe the dynamics of the implementation of these states, one needs to use a stochastic equation.

The purpose of the article  is to show the possibility of studying random processes described by the stochastic differential It\^{o} equations, the properties of which can jump-like change at random times, using the method of indicator random processes.

In our previous article \cite{ref-proceeding}, we proposed a method and an idea for its use in modeling diffusion with delay centers. In this article, we consider several models of random processes associated with diffusion processes. For the presented models of random processes, characteristic functions have been constructed that determine the probabilistic properties of these processes. The application of the proposed method of indicator random processes for finding solutions to equations for characteristic functions is demonstrated.

\section{Indicator random process}

Indicator functions are often used to identify subsets, points, and properties from a certain set.
An indicator function in the form of a "delta function"\, is used, for example, to make the transition from the original, generally nonlinear dynamic system to an equivalent description in terms of linear partial differential equations% \cite{ref-book11}
.

Let us introduce the concept of an indicator random process. Here and below, all random processes under consideration are defined on a common probability space.

\begin{dfn}\label{df1}
A random process $\chi(t)$ is called an indicator random process if it takes only two values: 1 or 0.
% We call an indicator random process a jump-like random process $\chi(t)$, which can take only two values: 1 or 0.
\end{dfn}

\noindent

To study the properties of the process $\chi(t)$, we use following notations: $Prob(A)$ is the probability of an event A, $E[\chi(t)]$ is the mathematical expectation of a random process $\chi(t)$, $t \geqslant 0$ is a time.

\begin{prp}\label{Proper1}
For any time $t$ the following conditions are satisfied:
%\begin{linenomath}
\begin{equation}\label{1}
(\chi(t))^{\alpha} = \chi(t), \ \ \ (1-\chi(t))^{\alpha} = 1-\chi(t) \ \ \  \forall\,\alpha>0.
\end{equation}
%\end{linenomath}
\end{prp}

\begin{prp}\label{Proper2}
The following equalities are satisfied:
%\begin{linenomath}
\begin{equation}\label{2}
  \begin{gathered}
    E[\chi(t)] = Prob(\chi(t)=1), \\
    Prob(\chi(t)=0) = 1 - E[\chi(t)].
  \end{gathered}
\end{equation}
%\end{linenomath}
\end{prp}

%\begin{Property}\label{Proper3}
%\hl{For any positive degree $\alpha$ it is hold}:
%%\begin{linenomath}
%\begin{equation*}%\label{3}
%(1-\chi(t))^{\alpha} = 1-\chi(t).
%\end{equation*}
%%\end{linenomath}
%\end{Property}
%
%\noindent
%\hl{As an example, we can prove the case where $\alpha\in \mathbf{N}$ using the binomial expansion}.

\begin{dfn}\label{df14}
Two and more random processes $\xi_{1}(t),\xi_{2}(t), ... $, are called incompatible processes if for every $t\geqslant 0$ only one process from this collection is nonzero:
%\begin{linenomath}
\begin{equation}\label{4==}
\begin{array}{c}
  \xi_{k_{1}}(t_{1})\neq 0, \ \ \ \
  \xi_{j}(t_{1})= 0, \ \ for \ any \ j\neq k_{1}; \\
 \xi_{k_{2}}(t_{2})\neq 0, \ \ \ \
  \xi_{j}(t_{2})= 0, \ \ for \ any \ j\neq k_{2}; \\
  \ldots
\end{array}
\end{equation}
%\end{linenomath}
\end{dfn}

\begin{Lem}\label{lem2}
Let $\chi_{j}(t)$, $j=1,2,\ldots,n-1$, be independent indicator random processes. Then the random processes
%\begin{linenomath}
\begin{equation}\label{4}
  \begin{gathered}
    z_{1}(t) = \chi_{1}(t); \ \ \ z_{k}(t) = \chi_{k}(t)\prod\limits_{j=1}^{k-1}(1-\chi_{j}(t)), \ \ \ k=2,3,\ldots,n-1; \\
    z_{n}(t) = \prod\limits_{j=1}^{n-1}(1-\chi_{j}(t)),
  \end{gathered}
\end{equation}
%\end{linenomath}
form a complete group of incompatible processes for every $t \geqslant 0$.
\end{Lem}

\begin{proof}
1. Each of the processes \eqref{4} can take only two values: 1 and 0. Let us consider the product of random processes $z_{l}(t)z_{k}(t)$ at any time $t$ for any $1 \leqslant k,l<n$. Taking into account the Definition \ref{df1} of an indicator random process and Property \ref{Proper1}, we obtain
%\begin{linenomath}
\begin{equation}\label{5}
  \begin{aligned}
    z_{l}(t)z_{k}(t) & = \chi_{k}(t)\prod\limits_{j=1}^{k-1}(1-\chi_{j}(t))\chi_{l}(t)\prod\limits_{j=1}^{l-1}(1-\chi_{j}(t)) \\
    & = \chi_{k}(t)(1-\chi_{k}(t)) \prod\limits_{j=1}^{k-1}(1-\chi_{j}(t))^{2}\chi_{l}(t)\prod\limits_{i=k+1}^{l-1}(1-\chi_{i}(t))=0.
  \end{aligned}
\end{equation}
%\end{linenomath}
Similarly, for any $1 \leqslant k<n$ the following holds: $z_{k}(t)z_{n}(t)=0$.
Therefore, random processes \eqref{4} are incompatible (see Definition \ref{df14}).

2. Consider the process $Z_{ n }(t)=\sum\limits_{j=1}^{n} z_{n}(t)$:
%\begin{linenomath}
\begin{equation}\label{7}
  \begin{aligned}
    Z_{ n }(t) & = \chi_{1}(t)+ \sum\limits_{k=2}^{n-1} \chi_{k}(t)\prod\limits_{j=1}^{k-1}(1-\chi_{j}(t))+ \prod\limits_{j=1}^{n-1}(1-\chi_{j}(t)) \\
    & = \chi_{1}(t)+(1-\chi_{1}(t))[\chi_{2}(t)+ \chi_{3}(t)(1-\chi_{2}(t)) + \ldots \\
    & ~~~ {} + \chi_{n-1}(t)(1-\chi_{2}(t))\cdots(1-\chi_{n-2}(t))+ (1-\chi_{2}(t))\cdots(1-\chi_{n-1}(t))].
  \end{aligned}
\end{equation}
%\end{linenomath}
At any time $t$, the process \eqref{7} will take the value 1 (due to the Definition \ref{df1} of processes $\chi_{j}(t)$). Thus, at any given time $t$, only one of the processes \eqref{4} will take the value 1.

The results \eqref{5}, \eqref{7} lead to the statement of the lemma.
\end{proof}

\begin{Rem}
Using a given set of incompatible processes that form the complete group, and knowing their probabilities, one can proceed to constructively specifying the implementations of random processes with mixture structure (see Section \ref{appl2}).
\end{Rem}

\section{Characteristic function for a sum of incompatible processes}
Let $g_{1}(t), \ldots, g_{n}(t)$ be random processes. Consider the random process $Y(t)=\sum\limits_{j=1}^{n}y_{j}(t)$, where
%\begin{linenomath}
\begin{equation*}%\label{9}
  \begin{gathered}
    y_{1}(t)=\chi_{1}(t)g_{1}(t); \ \ \ y_{k}(t)=\chi_{k}(t)\prod\limits_{j=1}^{k-1}(1-\chi_{j}(t))g_{k}(t), \ \ \ k=2,3,\ldots,n-1; \\
    y_{n}(t)=\prod\limits_{j=1}^{n-1}(1-\chi_{j}(t))g_{n}(t),
  \end{gathered}
\end{equation*}
%\end{linenomath}
and indicator random processes $\chi_{j}(t)$, $j=1,2,\ldots,n$, are independent of each other and with the random processes $g_{l}(t)$, $l=1,\ldots,n$. Then the random process $Y(t)$ has the form:
%\begin{linenomath}
\begin{equation}\label{10}
   Y(t)=\chi_{1}(t)g_{1}(t)+ \sum\limits_{k=2}^{n-1}\chi_{k}(t)\prod\limits_{j=1}^{k-1}(1-\chi_{j}(t))g_{k}(t)  +\prod\limits_{j=1}^{n-1}(1-\chi_{j}(t))g_{n}(t).
\end{equation}
%\end{linenomath}

In accordance with Lemma \ref{lem2}, at any time $t$ only one of the processes $g_{k}(t), \ k=1,\ldots, n,$ will be accomplished.

Let us construct the characteristic function of the random process \eqref{10}.
\begin{Thm}\label{th1}
Let $\chi_{j}(t)$, $j=1,2,\ldots,n-1$, be independent indicator random processes, and $Prob(\chi_{i}(t)=1)=p_{i}(t)$. Then the characteristic function of the process \eqref{10} has the form:
%\begin{linenomath}
\begin{equation*}%\label{101}
 J(t)=p_{1}(t)E[e^{i\beta g_{1}(t)}]+ \sum\limits_{k=2}^{n-1}E[e^{i\beta g_{k}(t)}]p_{k}(t)\prod\limits_{j=1}^{k-1}(1-p_{j}(t))  +E[e^{i\beta g_{n}(t)}]\prod\limits_{j=1}^{n-1}(1-p_{j}(t)).
\end{equation*}
%\end{linenomath}
\end{Thm}

\begin{proof} Construct the characteristic function for the process \eqref{10}:
 %\begin{linenomath}
\begin{equation*}%\label{102}
  \begin{aligned}
    J(t) & = E[e^{i\beta Y(t)}] \\
    & = E \biggl[ \exp \biggl\{ i\beta\left(\chi_{1}(t)g_{1}(t)+ \sum\limits_{k=2}^{n-1}\chi_{k}(t)\prod\limits_{j=1}^{k-1}(1-\chi_{j}(t))g_{k}(t)+\prod\limits_{j=1}^{n-1}(1-\chi_{j}(t))g_{n}(t) \right)\biggl\} \biggr] \\
    & = E[e^{i\beta g_{1}(t)}] \prod\limits_{k=2}^{n-1} \exp \biggl\{ i\beta\chi_{1}(t)\prod\limits_{j=1}^{k-1}(1-\chi_{j}(t))g_{k}(t) \biggr\} \exp \biggl\{ i\beta \prod\limits_{j=1}^{n-1}(1-\chi_{j}(t))g_{n}(t) \biggr\}.
  \end{aligned}
\end{equation*}
%\end{linenomath}

Further, we apply the exponential series expansion, taking into account Lemma \ref{lem2}, Property \ref{Proper2}, and the mutual independence of $g_{k}(t)$ and $\chi_{l}(t)$ for any indices $k,l$. Next, we calculate the mathematical expectation, and taking into account \eqref{2}, we obtain the statement of the theorem.
\end{proof}

\section{Application of indicator random processes}

Let us examine several interesting examples of using indicator random processes to construct mathematical models of physical processes which associated with diffusion processes.

\subsection{Processes in an environment with jump-like changes in properties}

Using the properties of indicator random processes $\chi_{1}(t)$ and $\chi_{2}(t)$, we can construct an equation for dynamic processes with jump-like changes in properties. Let us consider the system of It\^o stochastic differential equations:
%\begin{linenomath}
\begin{equation*}%\label{105}
  \begin{aligned}
    dx(t) & = \chi_{1}(t) a_{1}(x(t),t)dt+\chi_{2}(t)B_{1}(x(t),t)d{\rm w}(t) \\
    & ~~~ {} + (1-\chi_{1}(t)) a_{2}(x(t),t)dt+(1-\chi_{2}(t))B_{2}(x(t),t)d{\rm w}(t),
  \end{aligned}
\end{equation*}
%\end{linenomath}
where, in the general case, $x(t), \, a_{j}(x(t),t)\in \mathbf{R}^{n}$, ${\rm w}(t)$ is the $m$-dimensional Wiener process with independent components, $B_{j}(x(t),t)$ is the matrix of size $n\times m$, $j=1,2$.

Using the indicator random process $\chi(t)$, one can also construct a model of the diffusion process with transitions from one subspace to another:
%\begin{linenomath}
\begin{equation}\label{106}
  \begin{aligned}
    dx(t) & = \chi(t)[a_{1}(x(t),y(t),t)dt+ B_{1}(x(t),y(t),t)d{\rm w}(t)],\\
    dy(t) & = (1-\chi(t))[a_{2}(x(t),y(t),t)dt+ B_{2}(x(t),y(t),t)d{\rm w}(t)],
  \end{aligned}
\end{equation}
%\end{linenomath}
etc. Such problems  arise in the course of simulating a diffusion process with a	 non-random modulus of speed, when the magnitude of the velocity modulus can change abruptly under the influence of external random disturbances, and remain constant between these jumps~\cite{ref-book6}.

If the coefficients of \eqref{106} satisfy to conditions
%\begin{linenomath}
\begin{equation*}%\label{106}
B_{j}(x(t),y(t),t)=0, \ \ \
a_{1}(x(t),y(t),t)=-a_{2}(x(t),y(t),t)=-c, \ \ \ c=const>0,
\end{equation*}
%\end{linenomath}
then the process $x(t)+y(t)$ corresponds to the Kac model of particle motion with random changes in the direction of velocity \cite{ref-book3}.

\subsection{Diffusion process with random time delay centers}

We propose a new version of the model of dynamic process with delay centers employs indicator random processes.

A delay center, or a time absorption center, arises, for example, for queuing processes: a device becomes a delay center when serving a customer.
For the diffusion process, such centers can be considered points in space where a particle temporarily stuck and leaves them at a random moment in time.

We will correlate the randomness of the moment of stopping the particle’s movement and the moment of resumption its movement with a random function of a non-random integer value $N(t)$ with independent increments (indicator random process) \cite{ref-proceeding}:
%\begin{linenomath}
\begin{equation*}%\label{11}
\chi(t)=\bar{\chi}(N(t)).
\end{equation*}
%\end{linenomath}

\noindent
\begin{dfn}\label{df2}
An indicator random function $\chi(t)$ is called conditionally periodic if it satisfies the condition
%\begin{linenomath}
\begin{equation}\label{12}
\chi(t)=\bar{\chi}(N(t)+2k),\ \ \ \ \ k\in\mathbf{N}.
\end{equation}
%\end{linenomath}
\end{dfn}

The conditions \eqref{1}, \eqref{12} are satisfied by the function
%\begin{linenomath}
\begin{equation}\label{13}
\chi(t)=\bar{\chi}(N(t))=  0.5 (1+ \cos[\pi N(t)]).
\end{equation}
%\end{linenomath}

\noindent
This function has the following properties:
%\begin{linenomath}
\begin{equation*}%\label{14}
  \chi(t)= \begin{cases}
      0.5 (1+ \cos[\pi N(t)])=1,& \text{for} \ \ N(t)=2s, \ \ s\in \mathbf{N}\cup\{0\}, \\
      0.5 (1+ \cos[\pi N(t)])=0, & \text{for} \ \ N(t)=2s+1, \ \ s\in \mathbf{N}\cup\{0\}.
  \end{cases}
\end{equation*}
%\end{linenomath}
For example, as $N(t)$ we can take a homogeneous Poisson process: $E[N(t)]=\mu t$. Then $E[\bar{\chi}(N(t))]= 0.5(1 + \exp\{-2\mu t\})$.

To model diffusion processes that have centers of random time delay (diffusion interruption) and subsequent restoration of the movement process, let us use the following It\^o stochastic differential equation:
%\begin{linenomath}
\begin{equation}\label{15}
dx(t)=\chi(t)[a(x(t),t)dt+B(x(t),t)d{\rm w}(t)],
\end{equation}
%\end{linenomath}
where, in the general case, $x(t), \, a(x(t),t)\in \mathbf{R}^{n}$, ${\rm w}(t)$ is the $m$-dimensional Wiener process with independent components, $B(x(t),t)$ is the matrix of size $n\times m$.

As is known,  diffusion occurs by several mechanisms.
Surface diffusion is a general process involving the motion of molecules, and atomic clusters at solid material surfaces, and the corresponding model is the system of equations \eqref{15} in $\mathbf{R}^{2}$. Bulk diffusion, i. e. diffusion in the bulk of the material, can be modelled  by the system of equations \eqref{15} in $\mathbf{R}^{3}$.

Let us consider the equation
%\begin{linenomath}
\begin{equation}\label{1006}
dx(t)=a(t)dt+  0.5 (1+ \cos[\pi N(t)])b(t)d{\rm w}(t), \ \ \ x(0)=0.
\end{equation}
%\end{linenomath}
This model corresponds to the case when the state of the system changes according to a deterministic law, and then, over a random period of time, it is affected by random disturbances. Let us make a change of variables:
%\begin{linenomath}
\begin{equation*}%\label{1007}
y(t)=x(t)-\displaystyle \int_{0}^{t}a(\tau)d\tau.
\end{equation*}
%\end{linenomath}
Then the equation \eqref{1006} takes the form
%\begin{linenomath}
\begin{equation}\label{1008}
dy(t)=  0.5 (1+\cos[\pi N(t)])b(t)d{\rm w}(t),
\end{equation}
%\end{linenomath}
where $y(t), \, a(t), \, b(t), \,\in \mathbf{R}$, $ {\rm w}(t) $ is one-dimensional Wiener process, $N(t)$ is Poisson process, and $N(t)$ and ${\rm w}(t)$ are mutually independent ones. Suppose, that $y(0)=0$.

The characteristic function for the random process \eqref{1008} is
%\begin{linenomath}
\begin{equation*}%\label{1008}
J_{1}(t) = E[e^{i\beta y(t)}] = E \biggl[ \exp \biggl\{ 0.5 i\beta \displaystyle \int_{0}^{t}(1+ \cos[\pi N(\tau)])b(\tau)d{\rm w}(\tau) \biggr\} \biggl].
\end{equation*}
%\end{linenomath}

\noindent
If, during the simulation, we take into account that the average time interval between a particle being at rest and in motion are different, then as $N(t)$ we can choose, for example, the random process $N(t)=N_{1}(t)N_{ 2}(t)$, where $N_{1}(t)$, $N_{2}(t)$ are independent Poisson processes for which $E[N_{1}(t)]=\mu_{1} t$, $E[N_{2}(t)]=\mu_{2} t$. Since the relation holds
%\begin{linenomath}
\begin{equation*}%\label{16}
  \begin{aligned}
    & \sum_{s=0}^{\infty}Prob(N(t)=2s) \\
    & ~~~ = 1-\sum_{s_{1}=0}^{\infty}Prob(N_{1}(t)=2s_{1}+1) \sum_{s_{2}=0}^{\infty}Prob(N_{2}(t)=2s_{2}+1),
  \end{aligned}
\end{equation*}
%\end{linenomath}
i.e., the probability of a particle in motion will be greater than the probability of being at rest.

If the situation is the opposite, then we proceed to using the process $$\hat{\chi}(t)=\chi(N(t)+1).$$

\noindent
Applying It\^o formula, taking into account \eqref{1}, and using the explicit form of the Poisson distribution \cite{ref-book2}, we obtain
%\begin{linenomath}
\begin{equation}\label{18}
  \begin{aligned}
    D(t) = \displaystyle\frac{dE\left[ y^{2}(t)\right]}{dt} & = 0.5 b^{2}  \langle 1+ \cos[\pi N(t)] \rangle \\
    & = \displaystyle\frac{1}{4} b^{2}[4- (1-e^{-2\mu_{1}t})(1-e^{-2\mu_{2}t})],
  \end{aligned}
\end{equation}
%\end{linenomath}
where $b^{2}$ is the diffusion coefficient without time delay centers.
It follows from equality \eqref{18} that
%\begin{linenomath}
\begin{equation*}%\label{19}
D(0)=b^{2}, \ \ \ D(\infty)=\displaystyle\frac{3}{4}b^{2}.
\end{equation*}
%\end{linenomath}

Note that the difference between the probabilities for even and odd $N(t)$ is due to the fact that $N(t)\in \mathbf{N}\cup \{0\}$ and $N(0)=0$, i.e., $N(t)$ does not start with an odd number. This leads to a discrepancy among the analytical expressions for the probabilities:
%\begin{linenomath}
\begin{equation*}%\label{22}
  \begin{aligned}
    \sum_{s=0}^{\infty}Prob(N(t) = 2s) & = e^{-\mu t}\sum_{s=0}^{\infty}\displaystyle\frac{(\mu t)^{2s}}{(2s)!} = e^{-\mu t}\cosh \mu t = 0.5 (1+e^{-2\mu t}), \\
    \sum_{s=0}^{\infty}Prob(N(t)=2s+1) & = e^{-\mu t}\sum_{s=0}^{\infty}\displaystyle\frac{(\mu t)^{2s+1}}{(2s+1)!} = e^{-\mu t}\sinh \mu t = 0.5 (1-e^{-2\mu t}) \\
    & = 1-\sum_{s=0}^{\infty}Prob(N(t)=2s).
  \end{aligned}
\end{equation*}
%\end{linenomath}
The equality of probabilities is only asymptotic:
%\begin{linenomath}
\begin{equation*}%\label{22}
\lim\limits_{t\to \infty}\sum_{s=0}^{\infty}Prob(N(t)=2s)=
 \lim\limits_{t\to \infty}\sum_{s=0}^{\infty}Prob(N(t)=2s+1) = 0.5 .
\end{equation*}
%\end{linenomath}
Thus, if in a homogeneous media there is a nonlinear time dependence of the average square displacement of a particle and an asymptotic decrease in the diffusion coefficient, we can conclude that there exist time delay centers.

\subsection{Diffusion with random change in direction of movement. The Kac model}

Let us consider one-dimensional particle movement with speed $v$, when the direction of movement changes at random times \cite{ref-book3}:
%\begin{linenomath}
\begin{equation*}%\label{202}
\left\{
  \begin{aligned}
    dx^{+}(t) & = \chi(t)vdt, \\
    dx^{-}(t) & = -(1-\chi(t))vdt.
  \end{aligned}
\right.
\end{equation*}
%\end{linenomath}

\noindent
Since we are  interested in the total displacement $x(t)=x^{+}(t)+x^{-}(t)$, we obtain the equation
%\begin{linenomath}
\begin{equation}\label{2012}
dx(t)=(2\chi(t)-1)vdt.
\end{equation}
%\end{linenomath}

For $\chi(t)$ we take the representation \eqref{13}:
$
\chi(t)=\bar{\chi}(N(t)) = 0.5 (1+ \cos[\pi N(t)]).
$
Let $v=const=c>0$.
Then the characteristic function for the process $x(t)$ takes the form:
%\begin{linenomath}
\begin{equation}\label{2008}
I(t) = E \biggl[ \exp \biggl\{i\beta \int_{0}^{t} \cos[\pi N(\tau)] c d\tau \biggr\} \biggr].
\end{equation}
%\end{linenomath}

\begin{Thm}\label{th2}
If $N(t)$ is the stationary Poisson process with parameter $\lambda$, the characteristic function for the process $x(t)$ subordinate to the system \eqref{2012} is a solution to the Cauchy problem:
%\begin{linenomath}
\begin{equation}\label{2009}
\displaystyle \frac{d^{2}I(t)}{dt^{2}}+ 2\lambda \frac{d I(t)}{dt}+ c^{2}\beta^{2}I(t)=0, \ \ \ I(0)=0, \ \ \  \frac{d I(0)}{dt}=ic\beta.
\end{equation}
%\end{linenomath}
\end{Thm}

\begin{proof}
Let us differentiate \eqref{2008}:
%\begin{linenomath}
\begin{equation}\label{2019}
  \frac{dI(t)}{dt} = E \biggl[ i\beta\cos[\pi N(t)]c\exp \biggl\{ i\beta\displaystyle \int_{0}^{t} \cos[\pi N(\tau)] c d\tau \biggr\} \biggr].
\end{equation}
%\end{linenomath}

\noindent
Let $f(x)\in \mathcal{C}^{\infty}$. Then
%\begin{linenomath}
\begin{equation}\label{20209}
df(N(t))=[f(N(t)+1)-f(N(t))]d N(t),
\end{equation}
%\end{linenomath}
where $dN(t)$ is an advanced increment, i.e., it is independent of the previous values of $N(t)$. Since $E[dN(t)]=\lambda dt,\, \lambda>0$, then taking into account \eqref{2019}, we obtain
%\begin{linenomath}
\begin{equation*}%\label{2020}
  \begin{aligned}
    d\frac{dI(t)}{dt} & = -c^{2}\beta^{2}E \biggl[ \cos^{2}[\pi N(t)] \exp \biggl\{ i\beta \int_{0}^{t}\cos[\pi N(\tau)]cd\tau \biggr\} \biggr]dt \\
    & ~~~ {} + i\beta c E \biggl[ (\cos[\pi(N(t)+1)]-\cos[\pi N(t)]) \exp \biggl\{ i\beta \int_{0}^{t}\cos[\pi N(\tau)] cd\tau \biggr\} \biggr] \lambda dt \\
    & = -c^{2}\beta^{2} E \biggl[ \exp \biggl\{ i\beta \int_{0}^{t}\cos[\pi N(\tau)]cd\tau \biggr\} \biggr] dt \\
    & ~~~ {} + 2 i\beta c \lambda E \biggl[ \cos[\pi N(t)] \exp \biggl\{ i\beta \int_{0}^{t}\cos[\pi N(\tau)]cd\tau \biggr\} \biggr] dt \\
    & = -c^{2}\beta^{2}I(t)dt-2\lambda \frac{dI(t)}{dt} dt.
  \end{aligned}
\end{equation*}
%\end{linenomath}
From the last equality we obtain the statement of the theorem.
\end{proof}

\begin{Rem} As is known, the characteristic function allows one to find the distribution density function. Applying the inverse Fourier transform to the equation \eqref{2009} from the Theorem \ref{th2}, we obtain the telegraph equation for the distribution density function $\rho(x,t)$:
%\begin{linenomath}
\begin{equation*}%\label{2029}
\frac{\partial^{2}\rho(x,t)}{\partial t^{2}}+ 2\lambda \frac{\partial \rho(x,t)}{\partial t}- c^{2}\frac{\partial^{2}\rho(x,t)}{\partial x^{2}}=0,
\end{equation*}
%\end{linenomath}
which coincides with the results obtained \cite{ref-book3}.
\end{Rem}

Note that the Kac model is finding new applications. In particular, it is used to study the model of random colliding particles interacting with the infinite reservoir at a fixed temperature and chemical potential \cite{ref-book7}. This is the so-called thermostat problem, in which particles can leave the system towards the reservoir or enter the system from the reservoir at random times. Accordingly, the proposed random indicator process method can also be used to solve the thermostat problem.

\subsection{Diffusion model with random transitions from one subspace to another. 2-dimensional case}

Let us consider  the following diffusion model:
%\begin{linenomath}
\begin{equation}\label{1060}
\left\{
  \begin{aligned}
    dx(t) & = \tilde{\chi}(t)bd{\rm w}(t),\\
    dy(t) & = (1-\tilde{\chi}(t))bd{\rm w}(t),
  \end{aligned}
\right.
\end{equation}
%\end{linenomath}
where $b$ is the diffusion coefficient, i.e., at random moments of time the process occurs either in the space $x(t)$ or $y(t)$. Such models can describe the diffusion process in random porous media.
As a random process $\tilde{\chi}(t)$ we choose the representation
$
\bar{\chi}(t)=   0.5 (1+ \cos N(t)).
$
Then the characteristic function for the process $\{x(t),y(t)\}$, in accordance with \eqref{1060}, takes the form:
%\begin{linenomath}
\begin{equation*}%\label{2021}
  \begin{aligned}
    J(t) & = E \biggl[ \exp \biggl\{ i\alpha\displaystyle \int_{0}^{t}0.5(1+\cos N(\tau))bd{\rm w}(\tau) + i\beta\displaystyle \int_{0}^{t}0.5(1-\cos N(\tau)) bd{\rm w}(\tau) \biggr\} \biggr].
  \end{aligned}
\end{equation*}
%\end{linenomath}

\begin{Thm}\label{th3}
If $(N(t)/\pi)$ is the stationary Poisson process with parameter $\lambda$, $J(t)$ is a solution to the equation:
%\begin{linenomath}
\begin{equation}\label{2022}
\frac{d^{2}J(t)}{dt^{2}}+ 0.5(4\lambda+[\alpha^{2}+\beta^{2}]b^{2} \frac{d J(t)}{dt}+ 0.5(\lambda\alpha^{2}+0.5\alpha^{2}\beta^{2}b^{2}+\lambda\beta^{2})b^{2}\beta^{2}J(t)=0.
\end{equation}
%\end{linenomath}
\end{Thm}

\begin{proof}
Since the processes $N(t)$ and ${\rm w}(t)$ are independent, we obtain
%\begin{linenomath}
\begin{equation}\label{2023}
  \begin{aligned}
    J(t) & = E \biggl[ \exp \biggl\{ i \int_{0}^{t} 0.5 [(\alpha+\beta)+(\alpha-\beta)\cos N(t) ]bd{\rm w}(\tau) \biggr\} \biggr] \\
    & = E \biggl[ \exp \biggl\{ -\int_{0}^{t}2^{-3}[(\alpha+\beta)+(\alpha-\beta)\cos N(t) ]^{2}b^{2}d\tau \biggr\} \biggr] \\
    & = E \biggl[ \exp \biggl\{ -\int_{0}^{t}2^{-3}[(\alpha+\beta)^{2}+2(\alpha^{2}-\beta^{2})\cos N(t) + (\alpha-\beta)^{2}]b^{2}d\tau \biggr\} \biggr] \\
    & = E \biggl[ \exp \biggl\{ -\alpha^{2}\displaystyle \int_{0}^{t}2^{-2}(1+\cos N(t))b^{2}d\tau - \beta^{2} \int_{0}^{t}2^{-2}(1-\cos N(t))b^{2}d\tau \biggr\} \biggr].
  \end{aligned}
\end{equation}
%\end{linenomath}

\noindent
For compactness, let us denote by $f(t)$ the last expression under the mathematical expectation sign in \eqref{2023}:
%\begin{linenomath}
\begin{equation}\label{20230}
 f(t)= \exp \biggl\{ -\alpha^{2}\displaystyle \int_{0}^{t}2^{-2}(1+\cos N(t))b^{2}d\tau - \beta^{2} \int_{0}^{t}2^{-2}(1-\cos N(t))b^{2}d\tau\biggr\} .
\end{equation}
%\end{linenomath}

Next we get
%\begin{linenomath}
\begin{equation}\label{141a}
\displaystyle\frac{dJ(t)}{dt} =-E[f(t)\{ \alpha ^{2} 2^{-2} (1+\cos N(t) )b^{2} +\beta ^{2} 2^{-2} (1-\cos N(t) )b^{2} \} ].
\end{equation}
%\end{linenomath}

\noindent
Taking into account that $dN(t)$ is an advanced increment, i.e., it does not depend on the previous values of $N(t)$, for which, due to the properties of the Poisson distribution,
$$
E[d(N(t)/\pi)]=\lambda dt, \ \ \ \lambda >0.
$$
Let us calculate the differential (compare with \eqref{20209}):
$$
d \cos N(t) =[\cos(N(t)+\pi)-\cos N(t) ]d(N(t)/\pi) = -2\cos N(t) d(N(t)/\pi).
$$
The process $\cos N(t) $ have the following properties:
%\begin{linenomath}
\begin{equation*}%\label{41aa}
  \begin{gathered}
    \cos N(t) \cos N(t) \equiv 1, \\
    (1-\cos N(t))(1+\cos N(t) )=1-\cos^{2} N(t) \equiv 0, \\
    (1-\cos N(t))^{2} = 1-\cos N(t),\\
    (1+\cos N(t))^{2} = 1+\cos N(t).
  \end{gathered}
\end{equation*}
%\end{linenomath}

\noindent
Then
%\begin{linenomath}
\begin{equation*}%\label{41b}
  \begin{aligned}
    \frac{d^{2} J(t)}{dt^{2}} & = E[f(t)\{ 0.5 \alpha ^{2} b^{2} \cos N(t) -\beta ^{2} 2^{-2} 2b^{2} \cos N(t) \} ]\lambda \\
    & ~~~ {} + E[f(t)\{ \alpha ^{2} 2^{-2} (1+\cos N(t) )b^{2} +\beta ^{2} 2^{-2} (1-\cos N(t) )^{2} b^{4} \} ^{2} ] \\
    & = E[f(t)\{0.5  \alpha ^{2} b^{2} \cos N(t) -0.5 \beta ^{2}  b^{2} \cos N(t) \} ]\lambda \\
    & ~~~ {} + E[f(t)\{ \alpha ^{4} 2^{-4} (1+\cos N(t) )^{2} b^{4} +\beta ^{4} 2^{-4} (1-\cos N(t) )^{2} b^{4} \} ] \\
    & =\lambda E[f(t)\{ 0.5 \alpha ^{2} (1+\cos N(t) )b^{2} +0.5  \beta ^{2}  (1-(\cos N(t) ))b^{2} \} ] \\
    & ~~~ {} - 0.5  \lambda (\alpha ^{2} +\beta ^{2} )b^{2} J(t) \\
    & ~~~ {} + E[f(t)\{ \alpha ^{4} 2^{-3} (1+\cos N(t) b^{4} +\beta ^{4} 2^{-3} (1-\cos N(t) )b^{4} \} ].
  \end{aligned}
\end{equation*}
%\end{linenomath}
This imply that
%\begin{linenomath}
\begin{equation}\label{2026}
  \begin{aligned}
    \frac{\partial^{2} J(t)}{\partial t^{2}} & = -2\lambda\frac{\partial J(t)}{\partial t} - 0.5 \lambda(\alpha^{2}+\beta^{2})b^{2}    J(t) \\
    & ~~~ {} + E[f(t)\alpha^{4}2^{-3}(1+ \cos N(t) )b^{4}]+E[f(t)\beta^{4}2^{-3}(1- \cos N(t) )b^{4}].
  \end{aligned}
\end{equation}
%\end{linenomath}

Let us continue the transformation for the last terms in \eqref{2026}. For simplicity, let us examine each of them separately.

Considering the first term:
%\begin{linenomath}
\begin{equation*}%\label{41d}
  \begin{aligned}
    & E[f(t)\alpha ^{4} 2^{-3} (1+\cos N(t) )b^{4} ] \\
    & ~~~ = 0.5\alpha ^{2} b^{2} E[f(t)\{ \alpha ^{2} 2^{-2} (1+\cos N(t) )b^{2} +\beta ^{2} 2^{-2} (1-\cos N(t) )b^{2} \} ] \\
    & ~~~ ~~~ {} -\alpha ^{2} \beta ^{2} b^{4} E[f(t)2^{-3} (1-\cos N(t) \} ],
  \end{aligned}
\end{equation*}
%\end{linenomath}
and taking into account \eqref{141a}, we have
%\begin{linenomath}
\begin{equation}\label{41w}
  \begin{aligned}
    & E[f(t)\alpha ^{4} 2^{-3} (1+\cos N(t) )b^{4}  ] \\
    & ~~~ = -0.5 \alpha ^{2}  b^{2}\displaystyle \frac{dJ(t)}{dt} - \alpha ^{2} \beta ^{2} b^{4} E[f(t)2^{-3} (1-\cos N(t) \} ].
  \end{aligned}
\end{equation}
%\end{linenomath}

\noindent
Let us transform the second term:
%\begin{linenomath}
\begin{equation}\label{41ws}
  \begin{aligned}
    & E[f(t)\beta ^{4} 2^{-3} (1-\cos N(t) )b^{4} ] \\
    & ~~~ = 0.5  \beta ^{2} b^{2} E[f(t)\{ \alpha ^{2} 2^{-2} (1+\cos N(t) )b^{2} +\beta ^{2} 2^{-2} (1-\cos N(t) )b^{2} \} ] \\
    & ~~~ ~~~ {} -\alpha ^{2} \beta ^{2} b^{4} E[f(t)2^{-3} (1+\cos N(t) )] \\
    & ~~~ = -0.5  \beta ^{2}  b^{2} \frac{dJ(t)}{dt} -\alpha ^{2} \beta ^{2} b^{4} E[f(t)2^{-3} (1+\cos  N(t) )]
  \end{aligned}
\end{equation}
%\end{linenomath}

\noindent
Substituting the resulting representations \eqref{41w} and \eqref{41ws} into \eqref{2026}, we obtain:
%\begin{linenomath}
\begin{equation*}%\label{41s}
  \alpha ^{2} \beta ^{2} b^{4} E[f(t)2^{-3} (1-\cos N(t)  ]+\alpha ^{2} \beta ^{2} b^{4} E[f(t)2^{-3} (1+\cos N(t) ) ] = \alpha ^{2} \beta ^{2} b^{4} 2^{-2} J(t).
\end{equation*}
%\end{linenomath}
As a result, taking into account the notation $E[f(t)]= J(t)$, we obtain
%\begin{linenomath}
\begin{equation*}%\label{41t}
  \begin{aligned}
    \frac{d^{2} J(t)}{dt^{2} } & = -2\lambda \frac{dJ(t)}{dt} - 0.5  \lambda (\alpha ^{2} +\beta ^{2} )b^{2} J(t) \\
    & ~~~ {} - 0.5  \alpha ^{2} \beta ^{2}  b^{2} \displaystyle\frac{dJ(t)}{dt} -\alpha ^{2} \beta ^{2} b^{4} 2^{-2} E[f(t)] \\
    & = -(2\lambda +0.5  \alpha ^{2} \beta ^{2}  b^{2} )\displaystyle\frac{dJ(t)}{dt} - [ \alpha ^{2} \beta ^{2} b^{4} 2^{-2} + 0.5  \lambda (\alpha ^{2} +\beta ^{2} )b^{2} ]J(t)
  \end{aligned}
\end{equation*}
%\end{linenomath}
The statement of Theorem \ref{th3} is obtained.
\end{proof}

The equation \eqref{2022} can be solved. However, its construction will be cumbersome.
Using the result obtained in Theorem \ref{th3}, we find an explicit view of the characteristic function for a simpler process.
We will solve the equation for a simpler model, and then we will show how to move to the explicit form of solving the equation \eqref{2022}.

As such a simple model, consider the characteristic function \eqref{1008} for the diffusion model with delay centers:
%\begin{linenomath}
\begin{equation*}%\label{2009}
\frac{d^{2}I(t)}{dt^{2}}+ 2\lambda \frac{d I(t)}{dt}+ c^{2}\beta^{2}I(t)=0, \ \ \ I(0)=0, \ \ \  \frac{d I(0)}{dt}=ic\beta.
\end{equation*}
%\end{linenomath}

The equation for the characteristic function $J_{1}(t)$ should coincide with \eqref{2022} if we set $\alpha=0$. Formally, this corresponds to the following substitutions of coefficients in \eqref{2022}:
%\begin{linenomath}
\begin{equation}\label{2030}
  \begin{aligned}
    & 0.5 (4\lambda+[\alpha^{2}+\beta^{2}]b^{2})  & & \Rightarrow & & 0.5 (4\lambda + \beta^{2}b^{2}), \\
    & 0.5 (\lambda\alpha^{2} + 0.5\alpha^{2}\beta^{2} +\lambda\beta^{2})b^{2} & & \Rightarrow & & 0.5 \lambda\beta^{2}b^{2}.
  \end{aligned}
\end{equation}
%\end{linenomath}

\noindent
Accordingly, the characteristic function $J_{1}(t)$ for the diffusion model with delay centers is a solution to the equation:
%\begin{linenomath}
\begin{equation}\label{2032}
\frac{d^{2}J_{1}(t)}{dt^{2}}+ 0.5(\beta^{2} b^{2}+ 4\lambda) \frac{d J_{1}(t)}{dt}+ 0.5\lambda\beta^{2}b^{2}J_{1}(t)=0.
\end{equation}
%\end{linenomath}
Since the conditions are satisfied
%\begin{linenomath}
\begin{equation*}%\label{2033}
  \begin{aligned}
    & 2^{-2} (\beta^{2} b^{2}+ 4\lambda)^{2}- 2\lambda\beta^{2}b^{2} \\
    & ~~~ = 2^{-2}\beta^{4} b^{4}+2\beta^{2} b^{2}\lambda+4\lambda^{2}-2\beta^{2}b^{2}\lambda \\
    & ~~~ = 2^{-2}\beta^{4} b^{4}+4\lambda^{2}>0 \ \ \ \forall\,\lambda>0,
  \end{aligned}
\end{equation*}
%\end{linenomath}
then the solution to the equation \eqref{2032} (\cite{ref-book1}, p. 375, formula 235(a)) will be as follows:
%\begin{linenomath}
\begin{equation*}%\label{2034}
  \begin{aligned}
    J_{1}(t) & = C_{1}\exp\biggl\{ -0.5 t\displaystyle(\beta^{2} b^{2}+ 4\lambda)+t\sqrt{\frac{1}{4}\beta^{4} b^{4}+4\lambda^{2}} \biggr\} \\
    & ~~~ {} + C_{2}\exp\biggl\{ -0.5 t\displaystyle (\beta^{2} b^{2}+ 4\lambda)-t\sqrt{\frac{1}{4}\beta^{4} b^{4}+4\lambda^{2}} \biggr\}.
  \end{aligned}
\end{equation*}
%\end{linenomath}
Since $\displaystyle\frac{1}{4}(\beta^{2} b^{2}+ 4\lambda)^{2}>\frac{1}{4}\beta^{4} b^ {4}+4\lambda^{2}$, then the first and second terms decrease with increasing $t$, and therefore $\lim\limits_{t\to \infty}J_{1}(t)=0$.
Taking into account the initial conditions
%\begin{linenomath}
\begin{equation*}%\label{2035}
\frac{d J_{1}(t)}{dt}\biggl|_{t=0} = -0.5 \beta^{2} b^{2}, \ \ \  J_{1}(t)\bigl|_{t=0}\bigr.=1,
\end{equation*}
%\end{linenomath}
we obtain an equation for determining the constants:
%\begin{linenomath}
\begin{equation*}%\label{2036}
 C_{1}+C_{2} = 1,
\end{equation*}
%\end{linenomath}
%\begin{linenomath}
\begin{equation*}
  \begin{gathered}
    \begin{aligned}
      0.5 \beta^{2} b^{2} & = C_{1}\exp\biggl\{ -0.5 t (\beta^{2} b^{2}+ 4\lambda)+t\sqrt{\frac{1}{4}\beta^{4} b^{4}+4\lambda^{2}} \biggr\} \\
      & ~~~ {} + C_{2}\exp\biggl\{ -0.5 t (\beta^{2} b^{2}+ 4\lambda)-t\sqrt{\frac{1}{4}\beta^{4} b^{4}+4\lambda^{2}} \biggr\}, \\
      0.5 \beta^{2} b^{2} & = (1-C_{2})\exp\biggl\{ -0.5 t (\beta^{2} b^{2}+ 4\lambda)+t\sqrt{\frac{1}{4}\beta^{4} b^{4}+4\lambda^{2}} \biggr\} \\
      & ~~~ {} + C_{2}\exp\biggl\{-0.5 t (\beta^{2} b^{2}+ 4\lambda)-t\sqrt{\frac{1}{4}\beta^{4} b^{4}+4\lambda^{2}} \biggr\},
    \end{aligned} \\
    \end{gathered}
\end{equation*}
%\end{linenomath}
 %\begin{linenomath}
\begin{equation*}
    2\lambda = C_{2}\sqrt{\frac{1}{4}\beta^{4} b^{4}+4\lambda^{2}}.
\end{equation*}
%\end{linenomath}
Having solved  this equations system, we establish that
%\begin{linenomath}
\begin{equation*}%\label{2037}
  \begin{gathered}
    C_{2} = 2\lambda\biggl( \sqrt{\frac{1}{4}\beta^{4} b^{4}+4\lambda^{2}} \biggr)^{-1} = \biggl( \frac{1}{16}\lambda^{-2}\beta^{4} b^{4}+1 \biggr)^{-0.5},\\
    C_{1} = 1 - \biggl( \frac{1}{16}\lambda^{-2}\beta^{4} b^{4}+1 \biggr)^{-0.5}.
  \end{gathered}
\end{equation*}
%\end{linenomath}
Thus, the solution to the equation \eqref{2032} takes the form
%\begin{linenomath}
\begin{equation*}%\label{2040}
  \begin{aligned}
    J_{1}(t) & = \exp\biggl\{ -0.5 t\displaystyle (\beta^{2} b^{2}+ 4\lambda)+t\sqrt{\frac{1}{4}\beta^{4} b^{4}+4\lambda^{2}} \biggr\} \\
    & ~~~ {} + \biggl( \frac{1}{16}\lambda^{-2}\beta^{4} b^{4}+1 \biggr)^{-0.5}\biggl[ \exp\biggl\{ -0.5 t (\beta^{2} b^{2} +4\lambda)-t\sqrt{\frac{1}{4}\beta^{4} b^{4}+4\lambda^{2}} \bigl.\biggr\}\bigr.\bigr. \\
    & ~~~ ~~~ {} - \exp\biggl\{ -0.5 t\displaystyle (\beta^{2} b^{2}+4\lambda)+t\sqrt{\frac{1}{4}\beta^{4} b^{4}+4\lambda^{2}} \biggr\}\bigl.\biggr].
  \end{aligned}
\end{equation*}
%\end{linenomath}
Since the condition is satisfied
%\begin{linenomath}
\begin{equation*}%\label{2042}
\biggl( \frac{1}{16}\lambda^{-2}\beta^{4} b^{4}+1 \biggr)^{-0.5} < 1,
\end{equation*}
%\end{linenomath}
then $J_{1}(t)>0$ for all $t \geqslant 0.$

Using inverse substitutions of coefficients based on relations \eqref{2030}, we obtain the solution to the equation \eqref{2022}. Taking into account the relationship between characteristic functions and moments, we can find random moments for the processes under consideration.

\subsection{Application of a random indicator process to specify implementations of random processes with variable structure}{\label{appl2}}
\begin{Lem}%\label{lem4}
Let the following be given: a collection of independent set-events ${\rm A}_{j}$,\ \ $j=1,2,\ldots,n-1$ and a complete group of incompatible events $ { \rm B}_{j}$,\ \ $j=1,2,\ldots,n$:
$$
B_{1} ={\rm A}_{1}; \ \ \ \ \ \   {\rm B}_{j} ={\rm A}_{j} \bigcap \biggl( \, \bigcap _{k=1}^{j-1} \bar{{\rm A}}_{k} \biggr), \ \ \ j=2,3,\ldots,n-1; \ \ \ \ \ \ {\rm B}_{n} =\bigcap _{k=1}^{n-1} \bar{{\rm A}}_{k} .
$$
Let us assume that a set of probabilities is given:
$$
Prob({\rm B}_{j}), \ \ \ j=1,2,\ldots,n-1;\ \ \ \ \ \ Prob({\rm B}_{n})=1-\sum\limits_{j=1}^{n-1}Prob({\rm B}_{j}).
$$
Then it is possible to establish a one-to-one correspondence between the sets $Prob({\rm A}_{j})$ and $Prob({\rm B}_{r} )$, \ \ $r,j= 1,2, \ldots,n-1$.
\end{Lem}

\begin{proof}
Due to the independence of ${\rm A}_{j},\ \ j=1,2,\ldots,n-1$, we get the equalities:
$$
  \begin{aligned}
    Prob({\rm B}_{1}) & = Prob({\rm A}_{1}), \\
    Prob({\rm B}_{2}) & = Prob({\rm A}_{2}) Prob({\rm \bar{A}}_{1}), \\
    Prob({\rm B}_{3}) & = Prob({\rm A}_{3}) Prob({\rm \bar{A}}_{2})Prob({\rm \bar{A}}_{1}), \\
    & \ldots, \\
    Prob({\rm B}_{n-1}) & = Prob({\rm A}_{n-1})\prod\limits_{k=1}^{n-2}Prob({\rm \bar{A}}_{k})
  \end{aligned}
$$
or
$$
  \begin{aligned}
    Prob({\rm B}_{1}) & = Prob({\rm A}_{1}), \\
    Prob({\rm B}_{2}) & = Prob({\rm A}_{2}) (1-Prob({\rm  A}_{1})), \\
    Prob({\rm B}_{3}) & = Prob({\rm A}_{3})(1- Prob({\rm  A}_{2})) (1-Prob({\rm {A}}_{1})), \\
    & \ldots, \\
    Prob({\rm B}_{n-1}) & = Prob({\rm A}_{n-1}) \prod\limits_{k=1}^{n-2}(1-Prob({\rm {A}}_{k})).
  \end{aligned}
$$

Switching from one equality to another, we obtain
$$
  \begin{aligned}
    Prob({\rm B}_{1}) & = Prob({\rm A}_{1}), \\
    Prob({\rm B}_{2}) & = Prob({\rm A}_{2}) (1-Prob({\rm  A}_{1}))=Prob({\rm A}_{2}) (1-Prob({\rm  B}_{1})) \\
    & ~~~ \Rightarrow \ \ \ Prob({\rm A}_{2}) = \frac{Prob({\rm B}_{2})}{1-Prob({\rm  B}_{1})}, \\
    Prob({\rm B}_{3}) & = Prob({\rm A}_{3})(1- Prob({\rm  A }_{2}))(1-Prob({\rm {A}}_{1})) \\
    & ~~~ \Rightarrow \ \ \ Prob({\rm A}_{3}) = \frac{Prob({\rm B}_{3})}{1-Prob({\rm  B}_{1})-Prob({\rm  B }_{2})}, \\
    & ~~~ \ldots, \\
    Prob({\rm B}_{n-1}) & = Prob({\rm A}_{n-1})\prod\limits_{k=1}^{n-2}(1-Prob({\rm {A}}_{k})) \\
    & ~~~ \Rightarrow \ \ \ Prob({\rm A}_{n-1}) = \frac{Prob({\rm B}_{n-1})}{1-\sum\limits_{k=1}^{n-2}Prob({\rm  B }_{k})}.
  \end{aligned}
$$
Thus, a one-to-one correspondence established.
\end{proof}

Consider an example. Let $\chi_{j}(t)$ be independent indicator random processes. For example, $\chi_{j}(t)$ can be as follows:
%\begin{linenomath}
\begin{equation}\label{009}
\chi_{j}(t)=0.5(1-\cos(\pi N_{j}(t)),
\end{equation}
%\end{linenomath}
where $N_{j}(t)$ are independent Poisson processes with variable intensity $\lambda _{j}(t)$, and
$$
Prob(N_{j}(t)=m) = \frac{a_{j} ^{m}(t)}{m!} e^{-a_{j} (t)}, \ \ \ m=0,1,2,\ldots , \ \ \ a_{j}(t) = \int _{0}^{t}\lambda _{j} (\tau )d\tau.
$$

\noindent
Let the events ${\rm B}_{j}$, $j-1,\ldots, n$, be incompatible events that constitute a complete group. Let us assume that events ${\rm A}_{j}$ are associated with process $\chi_{j}(t)$, and events ${\rm B}_{j}$ are associated with process $N_{j} ( t)$.

Let the events ${\rm B}_{j}$, $j-1,\ldots, n$, be incompatible events that form a complete group. Let us assume that the event process ${\rm A}_{j}$ is associated with the process $\chi_{j}(t)$, and the event process ${\rm B}_{j} $ is associated with the probabilities $Prob_{t } ({\rm B}_{j} )=p_{j}(t)$.

Taking into account \eqref{009}, events ${\rm A}_{j} $ will correspond only to odd values of the process $N_{j}(t)$.
Therefore,
%\begin{linenomath}
\begin{equation*}%\label{7}
  \begin{aligned}
    Prob_{t} (\bar{{\rm A}}_{j}) & = \sum_{k=0}^{\infty}Prob(N_{j}(t)|N_{j}(t)=2k) \\
    & = e^{-a_{j}(t)} \sum\limits_{m=0}^{\infty} \frac{(a_{j}(t))^{2m}} {(2m)!} = e^{-a_{j}(t)} \cosh a_{j}(t) = 0.5  (1+e^{-2a_{j}(t)}).
  \end{aligned}
\end{equation*}
%\end{linenomath}

\noindent
Let us proceed  to comparing the distributions $Prob_{t} ({\rm A}_{j} )$ and $Prob_{t} ({\rm B}_{j} )$:
%\begin{linenomath}
\begin{equation*}%\label{70}
  \begin{gathered}
    0 < Prob_{t} ({\rm A}_{j}) = 0.5 (1-e^{-2a_{j}(t)}) = \frac{Prob_{t} ({\rm B}_{j} )}{1-\sum\limits_{k=1}^{j-1} Prob_{t} ({\rm B}_{k})} = \frac{p_{j}(t)}{1-\sum\limits_{k=1}^{j-1} p_{k} (t) }=p_{j} (\chi (t))<1, \\
   1-2p_{j} (\chi (t)) = e^{-2a_{j}(t)} = 1 - \frac{2p_{j} (t)}{1-\sum\limits_{k=1}^{j-1} p_{k}(t)} = \frac{ -\sum\limits_{k=1}^{j} p_{k}(t) - p_{j} (t)}{1-\sum\limits_{k=1}^{j-1} p_{k}(t)} \geqslant 0.
  \end{gathered}
\end{equation*}
%\end{linenomath}

\noindent
From this equality it follows:
%\begin{linenomath}
\begin{equation}\label{701}
a_{j}(t) = 0.5  \ln \left[\frac{1-\sum\limits_{k=1}^{j-1} p_{k} (t)}{1-\sum\limits_{k=1}^{j} p_{k} (t)]-p_{j} (t)}\right].
\end{equation}
%\end{linenomath}

\noindent
As the numerator in the equality \eqref{701} is positive,
$$
  1 - \sum\limits_{k=1}^{j-1} p_{k}(t) = \sum\limits_{k=j}^{n} p_{k}(t) \geqslant 0,
$$
it is necessary to establish the conditions when the denominator is positive:
%\begin{linenomath}
\begin{equation}\label{702}
  1 - \sum\limits_{k=1}^{j} p_{k}(t) \geqslant p_{j}(t) \ \ \ \Rightarrow  \ \ \
  \sum\limits_{k=j+1}^{n} p_{k}(t) \geqslant p_{j}(t) \ \ \ \forall\,j=1,\ldots,n-1.
\end{equation}
%\end{linenomath}

\noindent
The requirement follows from  \eqref{702}:
%\begin{linenomath}
\begin{equation*}%\label{703}
p_{j+1} (t) \geqslant p_{j} (t).
\end{equation*}
%\end{linenomath}
Such ranking is always possible, and therefore we consider it as the initial one.

Since $a_{j}(t) \geqslant 0$, the following inequality must hold:
$$
  1 - \sum\limits_{k=1}^{j-1} p_{k} (t) \geqslant 1-\sum\limits_{k=1}^{j} p_{k}(t)-p_{j}(t),
$$
which is always held:
$$
  \biggl[ 1-\displaystyle\sum\limits_{k=1}^{j} p_{k}(t) \biggr] + p_{j}(t) \geqslant \biggl[ 1-\sum\limits_{k=1}^{j} p_{k}(t) \biggr] - p_{j}(t) \ \ \ \ \Rightarrow \ \ \ \ 0 \geqslant -2p_{j}(t).
$$

%%%%%%%%%%%%%%%%%%%%%%%%%%%%%%%%%%%%%%%%%%
\section*{Conclusion}

The proposed method of indicator random processes using the method of characteristic functions allows us to consider both previously known models and new ones that have physical interpretations, such as diffusion processes.
Theorems have been proven in which equations for the characteristic functions of the random processes under consideration are obtained. In the future, these characteristic functions can be used to find the probabilistic properties of random processes.

Also we note that if a complete group of incompatible random processes is given, then it is established that there is a set of independent indicator random processes. Based on the latter, it is possible to construct a complete group of events whose distribution will coincide with the distribution of a given group of incompatible events.

\end{document}